\documentclass[onecolumn]{article}
\usepackage[margin=1in]{geometry}
\usepackage[dvipdfmx]{graphicx,color}
\usepackage{amsthm}
\usepackage{bm}
\usepackage{mathtools}
\usepackage{tcolorbox}
\usepackage{subfig}
\usepackage{ascmac}
\usepackage{comment}
\usepackage{mathrsfs}
\usepackage{amscd}
\usepackage{url}
\usepackage[all]{xy}
\usepackage{bbm}
\usepackage{tikz-cd}
\usepackage{array}
\usepackage{amsmath}
\usepackage{cancel}
\usepackage{here}
\usepackage{listings}
\usepackage{float}
\usepackage{verbatim}
\usepackage{lipsum} 
\usepackage{multicol} 
\usepackage{xcolor}
\usepackage{colortbl}
\usepackage{longtable}
\usepackage{multirow}
\usepackage{amssymb}
\usepackage{booktabs}
\usepackage{pdflscape}
\usepackage{graphicx} 
\usepackage{empheq}
\usepackage{enumitem}

\numberwithin{equation}{section}

\usepackage[dvipdfmx]{graphicx}
\usetikzlibrary{calc}

\tcbuselibrary{breakable}

\makeatletter
\newcommand{\address}[1]{\gdef\@address{#1}}
\newcommand{\btri}{\mathrel{\blacktriangleright}}
\newcommand{\printaddress}{\@address}
\makeatother
\newtheoremstyle{mytheoremstyle}
  {10pt}
  {10pt}
  {\itshape}
  {}
  {\bfseries}
  {}
  {1em}
  {\thmname{#1}\thmnumber{ #2}\thmnote{ (#3)}}

\theoremstyle{mytheoremstyle}

\newtheorem{thm}{Theorem}[section]
\newtheorem{prop}[thm]{Proposition}
\newtheorem{defi}[thm]{Definition}
\newtheorem{fac}[thm]{Fact}

\newtheorem{remark}[thm]{Remark}
\newtheorem{lem}[thm]{Lemma}
\newtheorem{cor}[thm]{Corollary}

\makeatletter
\renewenvironment{proof}[1][\proofname]{%
  \par\normalfont 
  \topsep6pt\relax 
  \trivlist\item[\hskip\labelsep\itshape #1\@addpunct{.}]\ignorespaces
}{%
  \qed\endtrivlist\@endpefalse
}
\makeatother
\DeclareMathOperator{\PSL}{\mathrm{PSL}}
\DeclareMathOperator{\AGaL}{\mathrm{A\Gamma L}}
\DeclareMathOperator{\AGL}{\mathrm{AGL}}
\DeclareMathOperator{\GF}{\mathrm{GF}}

\DeclareMathOperator{\Sz}{\mathrm{Sz}}
\DeclareMathOperator{\PrN}{\mathrm{PrN}}
\DeclareMathOperator{\PH}{\mathrm{PH}}
\DeclareMathOperator{\NPr}{\mathrm{NPr}}
\newcommand{\id}{\mathrm{id}}


\begin{document}
\title{Classification and lattice properties of pronormal subgroups in $\PSL(2,q)$, $J_{1}$, and $\Sz(q)$ for the specified values of $q$}

\author{
Yuto Nogata\\
\normalsize{Graduate School of Science and Technology, Hirosaki University}\\
\normalsize{3 Bunkyo-cho, Hirosaki, Aomori, 036-8560, Japan}\\
\normalsize{Email: h24ms113@hirosaki-u.ac.jp}
}
\date{Novenber 1, 2025}

\maketitle

\begin{center}
\textbf{Abstract}
\end{center}
\noindent
We complete the classification of pronormal subgroups in the projective special linear groups \(\PSL(2,q)\), the Suzuki groups of Lie type \(\Sz(q)\), and the first Janko group \(J_{1}\), for the same ranges of \(q\) as in \cite{ferrara2023-2,ferrara2024}. Building on those works, we settle the remaining cases under the same parameter conditions. For each of these finite simple groups, the family of pronormal subgroups is closed under joins but not under meets. If the meet operation is replaced by a suitable operation, the family becomes a lattice.

\begin{center}
\textbf{Main Theorem}
\end{center}

\noindent
\textbf{(I) \(\bm{\PSL(2,q)}\)}\quad
Assume \(q\) lies in the ranges specified in \cite{ferrara2023-2,ferrara2024}. Write \(q=p^{\,n}\) with \(p\) prime and \(n\ge 1\). Then the only non-pronormal subgroups are the elementary abelian \(p\)-subgroups \((\mathbb{Z}_{p})^{\,j}\) with \(1\le j<n\), together with the \(2\)-subgroups described in Corollary~\ref{cor:PSLの準正規部分群分類}.

\medskip
\noindent
\textbf{(II) \(\bm{J_{1}}\)}\quad
Every subgroup is pronormal, except for \(\mathbb{Z}_{2}\) and \((\mathbb{Z}_{2})^{\,2}\), as stated in Corollary~\ref{cor:J1の準正規部分群の分類}.

\medskip
\noindent
\textbf{(III) \(\bm{\Sz(q)}\)}\quad
Assume \(q=2^{\,2n+1}\) with \(2n+1\) prime, as in \cite{ferrara2024}. Every subgroup is pronormal except for \(2\)-subgroups. Among the \(2\)-subgroups, the only pronormal ones are \((\mathbb{Z}_{2})^{\,2n+1}\) and the Sylow \(2\)-subgroup, as in Corollary~\ref{cor:鈴木群の準正規部分群分類定理}.

\medskip
\noindent
In each case, the family of pronormal subgroups is closed under joins but not under meets. After replacing the meet by a suitable operation defined later, the family becomes a lattice, as summarized in Remark~\ref{rem:束まとめ}.

\begin{center}
\textbf{Keywords}
\end{center}
\begin{center}
Finite simple groups, Pronormal subgroups, Subgroup lattices,\\ First Janko group, Projective special linear groups, Suzuki groups of Lie type.
\end{center}

\section{Introduction}

Let \(G\) be a group. A subgroup \(H\subset G\) is \emph{pronormal} in \(G\) when \(H\) and \(H^{g}\) are conjugate inside \(\langle H,H^{g}\rangle\) for every \(g\in G\). P. Hall introduced this notion to extend the well-behaved conjugacy properties of normal subgroups and maximal subgroups to a broader class. Pronormality is now a central embedding property in finite group theory. The key advantage is local to global. It reduces questions about conjugacy in \(G\) to questions inside the smaller subgroup \(\langle H,H^{g}\rangle\). This makes pronormality an effective tool for the analysis of subgroup embeddings.

According to \cite{Vdovin2013}, normal subgroups and maximal subgroups of any group are pronormal. In finite groups, every Sylow \(p\)-subgroup is pronormal. In addition, \cite{revin2012,Vdovin2013} prove that Hall \(\pi\)-subgroups and Carter subgroups of solvable groups are pronormal. They also show that Hall \(\pi\)-subgroups of finite simple groups are pronormal.

In recent years, classifying pronormal subgroups in finite simple groups has become increasingly important. Works in this direction, including \cite{ferrara2023-2,ferrara2024}, classify groups in which every non-abelian subgroup is pronormal and groups in which every non-nilpotent subgroup is pronormal. Beyond finite simple groups, \cite{mitkari2023a} classifies pronormal subgroups of dihedral groups, and \cite{mitkari2024} classifies pronormal subgroups of dicyclic groups. However, the structure of pronormal subgroups for broader families of groups remains only partially classified. Classifying pronormal subgroups more generally is important for understanding group structure.

Similarly, recent works also address whether the family of pronormal subgroups forms a lattice. In \cite{mitkari2023b,mitkari2024}, four group families are studied. For dicyclic groups and for dihedral groups, the family forms a lattice. For the alternating groups and the symmetric groups, the family does not form a lattice. For finite simple groups, general results are not yet available.

According to \cite{ferrara2023-2}, the finite simple groups in which every non-abelian subgroup is pronormal are precisely \(J_{1}\) and \(\PSL(2,q)\) for the values of \(q\) specified there. According to \cite{ferrara2024}, the finite simple groups in which every non-nilpotent subgroup is pronormal are \(J_{1}\), \(\PSL(2,q)\) for the values of \(q\) specified there, and \(\Sz(q)\) with \(q=2^{\,2n+1}\) and \(2n+1\) prime. The conditions on \(q\) for \(\PSL(2,q)\) differ between \cite{ferrara2023-2} and \cite{ferrara2024}.

Motivated by these results, we determine pronormality in the remaining cases for \(\PSL(2,q)\) and \(\Sz(q)\), where \(q\) lies in the ranges specified in \cite{ferrara2023-2,ferrara2024}. We also settle the case of \(J_{1}\). In particular, we give a complete classification of abelian and nilpotent pronormal subgroups in \(J_{1}\), \(\PSL(2,q)\), and \(\Sz(q)\) under those parameter conditions. Building on this classification, we examine whether the family of all pronormal subgroups forms a lattice.

In \S2, we collect preliminaries and notation. We treat \(\PSL(2,q)\) in \S3, \(J_{1}\) in \S4, and \(\Sz(q)\) in \S5, where \(q\) is restricted as in \cite{ferrara2023-2,ferrara2024}.

\section{Notation and Preliminaries}

Throughout this paper, we consistently use the notation $A^b := b^{-1}Ab$ for conjugation.

\begin{defi}\label{defi:準正規部分群の定義}
A subgroup \(H\subset G\) is called \emph{pronormal} in \(G\) if
\[
\forall\,g\in G,\;\exists\,x\in\langle H,H^{g}\rangle
\quad\text{s.t.}\quad
H^{x}=H^{g}.
\]
\end{defi}

\begin{lem}[\cite{revin2012,Vdovin2013}]\label{lem:準正規部分群のクラス}
Let \(G\) be a group. Then normal subgroups and maximal subgroups of \(G\) are pronormal. If \(G\) is finite, then for every prime \(p\), each Sylow \(p\)-subgroup of \(G\) is pronormal. If \(G\) is finite solvable, then every Hall \(\pi\)-subgroup and every Carter subgroup of \(G\) is pronormal. If \(G\) is finite simple, then every Hall \(\pi\)-subgroup of \(G\) is pronormal.
\end{lem}

\begin{defi}
A group $G$ is called \emph{prohamiltonian} if every non-abelian subgroup of $G$ is pronormal in $G$.
\end{defi}

\begin{thm}[\cite{ferrara2023-2}]\label{thm:プロハミルトン群分類}
 Let $G$ be a non-abelian finite simple group. Then $G$ is prohamiltonian if and only if it is isomorphic with one of the following groups:
 
\begin{enumerate}
  \item[(1)] $\PSL(2,q)$, where $q$ satisfies one of the following properties:
  \begin{enumerate}
    \item $q=2^{n}$ and $n$ is prime,
    \item $q=3^{n}$ and $n$ is an odd prime,
    \item $q=p$ is a prime such that $q\not \equiv \pm 1 \pmod {8}$ and $q>17$,
    \item $q=7,17$
  \end{enumerate}

  \item[(2)] $J_{1}$.
\end{enumerate}

\end{thm}

\begin{defi}
A group $G$ is called \emph{NPr-group} if every non-nilpotent subgroup of $G$ is pronormal in $G$.  
\end{defi}

\begin{thm}[{\cite{ferrara2024}}]\label{thm:非冪零ならば準正規な単純群分類}
Let $G$ be a non-abelian finite simple group. Then $G$ is NPr-group if and only if it is isomorphic to one of the following groups.

\begin{enumerate}
  \item[(1)] $\PSL(2,q)$, where $q$ satisfies one of the following properties:
  \begin{enumerate}
    \item $q=2^{n}$ and $n$ is prime,
    \item $q=3^{n}$ and $n$ is an odd prime,
    \item $q=p$ is prime and if $q\equiv \pm 1 \pmod {8}$, then either $q-1$ or $q+1$ is a power of $2$,
  \end{enumerate}

  \item[(2)] $\Sz(q)$, where $q=2^{2n+1}$ and $2n+1$ is a prime number,

  \item[(3)] $J_{1}$.
\end{enumerate}

\end{thm}

Based on Theorems \ref{thm:プロハミルトン群分類} and \ref{thm:非冪零ならば準正規な単純群分類}, this paper considers three classes of finite simple groups. The first class is $\PSL(2,q)$ with $q$ satisfying Theorems \ref{thm:プロハミルトン群分類} and \ref{thm:非冪零ならば準正規な単純群分類}. The second class is the Janko group $J_{1}$. The third class is the Suzuki group of Lie type $\Sz(2^{2n+1})$ with $2n+1$ prime. We aim to provide a complete classification of pronormal subgroups for each of these groups.

The discussion is organized in three sections. \S~3 treats $\PSL(2,q)$ under the conditions of Theorems \ref{thm:プロハミルトン群分類} and \ref{thm:非冪零ならば準正規な単純群分類}. \S~4 examines $J_{1}$. \S~5 analyzes $\Sz(2^{2n+1})$ with $2n+1$ prime. At the beginning of each section we fix the notation that is specific to the case under consideration. To accomplish the classification we rely on results from previous research.

\begin{defi}\label{defi:csc群}
A finite group $G$ is called \emph{csc-group} if given two cyclic subgroups $X, Y$ of $G$ of the same order, then there exists $g\in G$ such that $X=Y^{g}$. In other words, it refers to a group $G$ where all cyclic subgroups of the same order are conjugate to each other.
\end{defi}

\begin{lem}[\cite{Jabara2009},\S3]\label{lem:PSL,J1,SzはCSC群}
$\PSL(2,q)$ with $q\ge 3$, $\Sz(q)$ with $q\ge 8$, and $J_{1}$ are csc-groups. 
\end{lem}

\begin{lem}[\cite{revin2012}]\label{lem:中間のp部分群について}
Let $H$ be a subgroup of a group $G$. If $H$ contains a $p$-subgroup $P$ which is pronormal in $G$, then $H$ is pronormal in $G$ if and only if $H$ is pronormalized by every element of $N_{G}(P)$.
\end{lem}

\begin{cor}\label{cor:正規ならば準正規なので補題成立}
Let $H$ be a subgroup of a group $G$. If $H$ contains a $p$-subgroup $P$ which is pronormal in $G$.
If $H \triangleleft N_{G}(P)$, then $H$ is pronormal in $G$.
\end{cor}

\begin{proof}
If $H \triangleleft N_{G}(P)$, then by Lemma~\ref{lem:準正規部分群のクラス} the subgroup $H$ is pronormal in $N_{G}(P)$. 
By the equivalence in Lemma~\ref{lem:中間のp部分群について}, this implies that $H$ is pronormal in $G$.
\end{proof}

\begin{lem}[\cite{rose1967}]\label{lem:JohnRoseの補題}
Let \(G\) be a finite group and \(P\subset G\) a \(p\)-subgroup.
Then \(P\) is pronormal in \(G\) if and only if for every Sylow-\(p\) subgroup \(S\subset G\) with \(P\subset S\) one has \(P\triangleleft N_G(S)\).
\end{lem}

\begin{cor}[\cite{ferrara2023-2}]\label{cor:巡回シローp部分群の部分群も準正規}
Let \(G\) be a finite group whose Sylow-\(p\) subgroup \(S\) is cyclic. Then every subgroup \(H\subset S\) is pronormal in \(G\).
\end{cor}

\begin{lem}[\cite{DeGiovanniVincenzi2000}]\label{lem:準正規性の推移}
Let $\varphi:G\to G_{1}$ be a surjective homomorphism and let $H \subset G$ be pronormal in $G$.
Then $\varphi(H)$ is pronormal in $G_{1}$.
\end{lem}

\begin{lem}\label{lem:共役類の準正規性}
Let $H \subset G$ be a pronormal subgroup of $G$. Then any conjugate subgroup $H^a$ of $H$ is also a pronormal subgroup of $G$. Therefore, all subgroups conjugate to a pronormal subgroup $H$ in $G$ are pronormal.
\end{lem}

\begin{proof}
Apply Lemma~\ref{lem:準正規性の推移} to the automorphism
$\iota_{a}:G\to G$, $\iota_{a}(x)=a^{-1}xa$.
Since $\iota_{a}$ is surjective, Lemma~\ref{lem:準正規性の推移} yields that
$\iota_{a}(H)=H^{a}$ is pronormal in $G$. As $a\in G$ was arbitrary, the claim follows.
\end{proof}

\begin{cor}\label{cor:PSL,J1,Szの巡回部分群の準正規性は代表1つを見ればok}
Let $G$ be $\PSL(2,q)$ with $q\ge 3$, or $\Sz(q)$ with $q\ge 8$, or $J_{1}$, and let $H\subset G$ be cyclic of order $d$.
Then $H$ is pronormal in $G$ if and only if every cyclic subgroup of $G$ of order $d$ is pronormal.
In particular, pronormality and non-pronormality are uniform across all cyclic subgroups of order $d$ in $G$.
\end{cor}

\begin{proof}
By Lemma~\ref{lem:PSL,J1,SzはCSC群}, each of the listed groups is a csc-group, meaning that all cyclic subgroups of a fixed order are mutually conjugate. 
Therefore any two subgroups isomorphic to $\mathbb{Z}_{d}$ lie in a single conjugacy class. 
By Lemma~\ref{lem:共役類の準正規性}, pronormality is preserved under conjugation. 
Hence either every cyclic subgroup of order $d$ is pronormal or none is, proving the claim.
\end{proof}

\begin{lem}[\cite{DeGiovanniVincenzi2000}]\label{lem:直積しても準正規性は保存}
Let $G$ be a group and let $A,B\subset G$ be pronormal subgroups such that $AB=BA$. Then $AB$ is a pronormal subgroup of $G$.
\end{lem}

In this paper, we also discuss Frobenius groups. The definitions and properties of Frobenius groups are known as follows.
\begin{defi}\label{defi:フロベニウス群の定義と同値条件}
A group $F$ is called a \emph{Frobenius group} if it satisfies either of the following equivalent conditions.
\begin{equation}\label{eq:frob-def-1}
\parbox{0.92\linewidth}{
There exists a subgroup $H$ with $\{\id\}\subsetneq H \subsetneq F$ such that $H \cap H^{g}=\{\id\}$ for every $g\in F\setminus H$.
}
\end{equation}
\begin{equation}\label{eq:frob-def-2}
\parbox{0.92\linewidth}{
There exist a normal subgroup $K\triangleleft F$ and a subgroup $H\subset F$ with $F=K\rtimes H$. The conjugation action of $H$ on $K$ is fixed-point-free on $K\setminus\{\id\}$. This means $C_{K}(h)=\{\id\}$ for every $h\in H\setminus\{\id\}$.
}
\end{equation}

In this situation $H$ is called the \emph{Frobenius complement} and $K$ is called the \emph{Frobenius kernel}. In particular one has the semidirect decomposition $F=K\rtimes H$.
\end{defi}

\begin{remark}\label{rem:フロベニウスの部分群はフロベニウス}
Let \(F=K\rtimes H\) be a Frobenius group as in Definition \eqref{eq:frob-def-1} and \eqref{eq:frob-def-2}. Every subgroup \(F'\subset F\) can be written with subgroups \(A\subset K\) and \(B\subset H\) in the form \(F'=A\rtimes B\). In particular, for every \(a\in A\setminus\{\id\}\) one has \(C_{B}(a)=\{\id\}\). Hence \(F'\) satisfies \eqref{eq:frob-def-2} and is itself a Frobenius group. In particular, when both \(A\) and \(B\) are nontrivial, the semidirect product structure of \(F'\) does not collapse to the direct product \(A\times B\).
\end{remark}

\begin{lem}[{\cite{Robinson1996}}]\label{lem:フロベニウス部分群の正規部分群の分類}
Let $F=K\rtimes H$ be a Frobenius group and let $L\subset F$. Then the normal subgroups of $F$ are exactly those of the following two forms. No other normal subgroups occur.
\begin{flalign}
& \text{$H$-invariant normal subgroups of $K$.} && \label{eq:FrobClass1}\\
& \text{Those containing $K$ and corresponding to normal subgroups of the complement $H$.} && \label{eq:FrobClass2}
\end{flalign}
We say that $L$ is $H$-invariant if $L^{h}=L$ for all $h\in H$.
\end{lem}

\section{Pronormal Subgroups and Lattice Structure of $\bm{\PSL(2,q)}$}

This section concerns the finite simple group \(\PSL(2,q)\).
\begin{defi}
We use the following notation throughout this section.
{\allowdisplaybreaks
\begin{flalign*}
& \btri\; G := \PSL(2,q),\, q:=p^{n},\, o:=\gcd(q-1,2), &&\\
& \btri\; \PrN(G) := \{\,H \subset G \mid H \text{ is pronormal in } G\,\}, &&\\
& \btri\; v_{p'}(|G|): \text{the } p'\text{-adic valuation of } |G|=q(q-1)(q+1)/o.
\end{flalign*}
}
\end{defi}

We write \(\mathcal{Q}_{\PH}\) for the set of all parameters \(q\) such that \(G\) is prohamiltonian.
We write \(\mathcal{Q}_{\NPr}\) for the set of all parameters \(q\) such that \(G\) is an NPr-group.
We define the following subsets of \(\mathbb{Z}\) on the basis of Theorems~\ref{thm:プロハミルトン群分類} and \ref{thm:非冪零ならば準正規な単純群分類}.
\[
\mathcal{Q}_{2}=\{\,2^{n}\mid n \text{ is prime}\,\},\, 
\mathcal{Q}_{3}=\{\,3^{n}\mid n \text{ is an odd prime}\,\},\, \mathcal{P}_{\pm 3}=\{\,p\in\mathbb{Z}\mid p \text{ is prime},\ p\equiv \pm 3 \pmod{8},\ p>17\,\},
\]
\[\mathcal{E}=\{7,17\},\,
\mathcal{P}_{\pm 1}^{(2)}=\{\,p\in\mathbb{Z}\mid p \text{ is prime},\ p\equiv \pm 1 \pmod{8},\ \text{and } p-1 \text{ or } p+1 \text{ is a power of } 2,\, p\neq 7,17\}.
\]
\begin{alignat}{2}
\label{eq:PHとなるq全体}
& \mathcal{Q}_{\PH}  &\;=&\ \mathcal{Q}_{2}\ \cup\ \mathcal{Q}_{3}\ \cup\ \mathcal{E}\ \cup\ \mathcal{P}_{\pm 3}.\\
\label{eq:NPrとなるq全体}
& \mathcal{Q}_{\NPr} &\;=&\ \mathcal{Q}_{2}\ \cup\ \mathcal{Q}_{3}\ \cup\ \mathcal{P}_{\pm 1}^{(2)}.
\end{alignat}

We proceed according to this case division.

\subsection{Preliminaries for the Classification of Pronormal Subgroups in $\bm{\PSL(2,q)}$}

Regardless of the conditions on $q$, the subgroups of $G$ are classified as follows.

\begin{thm}[\cite{Dickson1958}]\label{thm:dicsonのPSL部分群分類定理}
Let $q:=p^{n}$. The subgroups of $G$ are precisely the following groups. 
{\allowdisplaybreaks
\begin{flalign}
& \btri\; \text{The dihedral group $D_{2d}$ of order $2d$ where $d\mid \tfrac{q\pm 1}{o}$ $(D_{2}\simeq \mathbb{Z}_{2}$ and $D_{4}\simeq (\mathbb{Z}_{2})^{2})$,} && \label{eq:dihedral}\\
& \btri\; \text{The cyclic group $\mathbb{Z}_{d}$ where $d\mid \tfrac{q\pm 1}{o}$,} && \label{eq:cyclic}\\
& \btri\; \text{$(\mathbb{Z}_{p})^{k}\rtimes \mathbb{Z}_{j}$, where $k\le n$, $j\mid p^{k}-1$, $j\mid \tfrac{q-1}{o}$,} && \label{eq:affine}\\
& \btri\; \text{$A_{4}$, except if $q=2^{e}$ with $e$ odd,} && \label{eq:A4}\\
& \btri\; \text{$S_{4}$, if $q\equiv \pm 1 \pmod {8}$,} && \label{eq:S4}\\
& \btri\; \text{$A_{5}$, except if $q\equiv \pm 2\pmod {5}$,} && \label{eq:A5}\\
& \btri\; \text{$\operatorname{PSL}(2,r)$, where $r$ is a power of $p$ such that $r^{m}=q$,} && \label{eq:PSL2r}\\
& \btri\; \text{$\operatorname{PGL}(2,r)$, where $r$ is a power of $p$ such that $r^{2m}=q$.} && \label{eq:PGL2r}
\end{flalign}
}
\end{thm}

\begin{prop}\label{prop:PSLで準正規性を求めるべき部分群}
Let $q=p^{n}$ and $o:=\gcd(q-1,2)$. Then the families of subgroups that may still require verification of pronormality reduce as follows.

\smallskip
\noindent\textup{(PH)} If $q$ satisfies \eqref{eq:PHとなるq全体}, then every $H\subset G$ is pronormal except possibly
\begin{equation}\label{eq:プロハミルトン条件q-EN}
\begin{gathered}
(\mathbb{Z}_{p})^{j}\ (1\le j< n),\qquad
\mathbb{Z}_{d}\ \bigl(d\mid \tfrac{q\pm 1}{o}\bigr),\qquad
(\mathbb{Z}_{2})^{2}\ \ (q\ \text{odd}).
\end{gathered}
\end{equation}

\noindent\textup{(NPr)} If $q$ satisfies \eqref{eq:NPrとなるq全体}, then every $H\subset G$ is pronormal except possibly
\begin{equation}\label{eq:NPr条件q-EN}
\begin{gathered}
(\mathbb{Z}_{p})^{j}\ (1\le j< n),\qquad
\mathbb{Z}_{d}\ \bigl(d\mid \tfrac{q\pm 1}{o}\bigr),\qquad
D_{2^{\,j}}\ \ (1\le j\le v_{2}(|G|)).
\end{gathered}
\end{equation}
\end{prop}

\begin{proof}
We treat the two cases separately.

\smallskip
\noindent\textbf{(PH).}
Assume $q$ satisfies \eqref{eq:PHとなるq全体}.
By Theorem~\ref{thm:プロハミルトン群分類}, every non-abelian subgroup of $G$ is pronormal. 
By Theorem~\ref{thm:dicsonのPSL部分群分類定理}, every abelian subgroup of $G$ has one of the following three forms:
\[
(\mathbb{Z}_{p})^{k}\ (0\le k\le n),\qquad 
\mathbb{Z}_{d}\ \bigl(d\mid \tfrac{q\pm 1}{o}\bigr),\qquad
(\mathbb{Z}_{2})^{2}\ \bigl(2\mid \tfrac{q\pm 1}{o}\bigr).
\]
The Sylow $p$-subgroup $(\mathbb{Z}_{p})^{n}$ is pronormal by Lemma~\ref{lem:準正規部分群のクラス}, hence only $1\le k<n$ may remain.
If $q$ is even, $(\mathbb{Z}_{2})^{2}$ is absorbed by $(\mathbb{Z}_{p})^{k}$ with $p=2$ and some $k\ge 2$.
If $q$ is odd, one of $q-1$ and $q+1$ is divisible by $4$, so $(\mathbb{Z}_{2})^{2}$ always occurs in $G$ and must be kept.
This is exactly the list in \eqref{eq:プロハミルトン条件q-EN}.

\smallskip
\noindent\textbf{(NPr).}
Assume $q$ satisfies \eqref{eq:NPrとなるq全体}.
Here $q$ is odd, so we fix $o:=\gcd(q-1,2)=2$.
By Theorem~\ref{thm:非冪零ならば準正規な単純群分類}, every non-nilpotent subgroup of $G$ is pronormal.
Therefore it suffices to list the nilpotent subgroups that can occur.
They are precisely the abelian ones already appearing in the (PH) case together with the $2$-groups of dihedral type.
We write these as $D_{2^{\,j}}$ with the convention $D_{2}\simeq \mathbb{Z}_{2}$ and $D_{4}\simeq (\mathbb{Z}_{2})^{2}$.

Such a subgroup occurs if and only if
$
 2^{j-1} | \frac{q\pm 1}{o}\,,
$
equivalently
$
 j \le v_{2}(q\pm 1).
$ Hence
$
 j_{\max} = \max\{v_{2}(q-1), v_{2}(q+1)\}.
$
Since $q$ is odd, one has $\min\{v_{2}(q-1),\,v_{2}(q+1)\}=1$, and therefore
$
 v_{2}(|G|) = v_{2}\!\left(\frac{q(q-1)(q+1)}{2}\right)
  = v_{2}(q-1)+v_{2}(q+1)-1.
$
Combining these equalities yields
\[
 j_{\max}
  = \max\{v_{2}(q-1),\,v_{2}(q+1)\}
  = v_{2}(q-1)+v_{2}(q+1)-1
  = v_{2}(|G|).
\]
\end{proof}

\begin{lem}[\cite{wilson2009}]\label{lem:PSLのシローp部分群とその正規化群}
Let $q:=p^{n}$ and put $o:=\gcd(q-1,2)$. For each prime $p'$ dividing $|G|$, choose $S\in\mathrm{Syl}_{p'}(G)$. Then $S$ and $N_{G}(S)$ are as follows.
{\allowdisplaybreaks
\begin{flalign}
& \btri\ \text{When } p'=2 \text{ and } p=2\text{: } 
   S \simeq (\mathbb{Z}_{2})^{n},\ 
   N_{G}(S) \simeq (\mathbb{Z}_{2})^{n} \rtimes \mathbb{Z}_{2^{n}-1}. && \label{eq:case-a}\\
& \btri\ \text{When } p'=2 \text{ and } q \equiv \pm 1 \pmod{8}\text{: } 
   S \simeq D_{2^{j}} \text{ and } N_{G}(S)=S \text{ with } j=v_{2}(|G|). && \label{eq:case-b}\\
& \btri\ \text{When } p'=2 \text{ and } q \equiv \pm 3 \pmod{8}\text{: }
   S \simeq V_{4},\ N_{G}(S) \simeq A_{4}. && \label{eq:case-c}\\
& \btri\ \text{When } p'=p\text{: } 
   S \simeq (\mathbb{Z}_{p})^{n},\ 
   N_{G}(S) \simeq (\mathbb{Z}_{p})^{n} \rtimes \mathbb{Z}_{(q-1)/o}. && \label{eq:case-2}\\
& \btri\ \text{When } p'\ne p,2 \text{ and } p' \mid \tfrac{q\pm 1}{o}\text{: } 
   S \simeq \mathbb{Z}_{p'^{\,j}} \text{ with } j=v_{p'}(|G|),\ 
   N_{G}(S) \simeq D_{2(q\pm 1)/o}. && \label{eq:case-3}
\end{flalign}
}
\end{lem}

By Lemma \ref{lem:PSLのシローp部分群とその正規化群}, when $p'=p$ the Sylow $p$-subgroup $S$ is elementary abelian and
\[
N_{G}(S)\simeq S\rtimes C \quad \text{with} \quad C\simeq \mathbb{Z}_{(q-1)/o}.
\]
The action of $C$ on $S$ is fixed–point–free on $S\setminus\{\id\}$. Therefore $N_{G}(S)$ is a Frobenius group whose Frobenius kernel is $S$ and whose Frobenius complement is $C$.

\begin{lem}[\cite{lidl1997,wilson2009}]\label{lem:規約性}
Let $q=p^{n}$. Let $S$ be a Sylow $p$-subgroup of $G$; then $S\simeq (\mathbb{Z}_{p})^{n}$, which we identify with the additive group of $\mathbb{F}_{q}$. Let $o:=\gcd(q-1,2)$ and let $C\subset N_{G}(S)$ be the cyclic complement with
$
N_{G}(S)=S\rtimes C$ and $C\simeq \mathbb{Z}_{(q-1)/o}
$.
Then $S$ is an irreducible $\mathbb{F}_{p}C$-module. Equivalently, the conjugation action of $C$ on $S$ is $\mathbb{F}_{p}$-linear and admits no nontrivial proper $C$-invariant $\mathbb{F}_{p}$-subspace of $S$.
\end{lem}

\begin{lem}[\cite{Robinson1996}]\label{lem:二面体群の部分群の正規性}
Let $D_{2d}=\langle r,s\mid r^{d}=\id,\ s^{2}=\id,\ srs=r^{-1}\rangle$ be the dihedral group of order $2d$. Assume $d\ge 3$. The following hold.
\begin{enumerate}
  \item Subgroups of order $2$.
  \begin{itemize}
    \item If $d$ is odd, all subgroups of order $2$ form a single conjugacy class and none of them is normal.
    \item If $d$ is even, there are three conjugacy classes of subgroups of order $2$. One class is the central subgroup $\langle r^{d/2}\rangle$, which is normal. The other two classes are generated by reflections and they are not normal.
  \end{itemize}

  \item Cyclic subgroups $\mathbb{Z}_{i}$ with $i\mid d$ and $i\neq 2$.
  \begin{itemize}
    \item Every such subgroup is contained in the normal cyclic subgroup $\langle r\rangle$. For each $i$ there is a unique subgroup of order $i$ and it is normal. All these subgroups form a single conjugacy class.
  \end{itemize}

  \item Dihedral subgroups $D_{2m}$ with $m\mid d$.
  \begin{itemize}
    \item If $d/m$ is odd, all such subgroups form a single conjugacy class.
    \item If $d/m$ is even, these subgroups split into two conjugacy classes.
    \item The normal dihedral subgroups are exactly the whole group $D_{2d}$ and, when $d$ is even, the subgroups of index $2$ in $D_{2d}$ which are isomorphic to $D_{2(d/2)}$. No other dihedral subgroup is normal.
  \end{itemize}
\end{enumerate}

For $d\in\{1,2\}$ the dihedral group $D_{2d}$ is abelian: $D_{2}\simeq \mathbb{Z}_{2}$ and $D_{4}\simeq \mathbb{Z}_{2}^{2}$. Hence every subgroup is normal.

\end{lem}

\subsection{Classification of Pronormal Subgroups of $\bm{\PSL(2,q)}$}

Assume $q=p^{n}$. Among the families listed in Proposition~\ref{prop:PSLで準正規性を求めるべき部分群}, we first determine the pronormality of $p$-subgroups of the form $(\mathbb{Z}_{p})^{j}$ with $1\le j<n$.

\begin{prop}\label{prop:小さな初等p部分群は非準正規}
Let $q=p^{n}$. Then every proper nontrivial elementary abelian $p$-subgroup $(\mathbb{Z}_{p})^{j}$ with $1\le j<n$ is not pronormal in $G$.
\end{prop}

\begin{proof}
Let $S$ be a Sylow $p$-subgroup of $G$, so $S\simeq (\mathbb{Z}_{p})^{n}$. Let $C$ denote the cyclic complement in $N_G(S)=S\rtimes C$ of order $(q-1)/o$. Fix $P\subset S$ with $\{\id\}\subsetneq P \subsetneq S$ and $|P|=p^{j}$ for some $1\le j<n$.

By Lemma~\ref{lem:JohnRoseの補題}, the subgroup $P$ is pronormal in $G$ if and only if $P\triangleleft N_G(S)$. In particular, $N_G(S)=S\rtimes C$ is a Frobenius group, hence by Lemma \ref{lem:フロベニウス部分群の正規部分群の分類}, and specifically by \eqref{eq:FrobClass1}, any normal subgroup contained in $S$ must be $C$-invariant. Therefore $P$ can be normal in $N_G(S)$ only if it is $C$-invariant.

By Lemma \ref{lem:規約性}, the conjugation action of $C$ on $S$ is irreducible over $\mathbb{F}_{p}$. Hence the only $C$-invariant subgroups of $S$ are $\{\id\}$ and $S$. Since $\{\id\}\subsetneq P \subsetneq S$, the subgroup $P$ is not $C$-invariant, and therefore $P\ntriangleleft N_G(S)$. Applying Lemma~\ref{lem:JohnRoseの補題} again, we conclude that $P$ is not pronormal in $G$.
\end{proof}

Assume $q=p^{n}$. By Proposition~\ref{prop:小さな初等p部分群は非準正規}, every elementary abelian $p$-subgroup $(\mathbb{Z}_{p})^{j}$ with $1\le j<n$ is non-pronormal. We therefore turn to cyclic subgroups $\mathbb{Z}_{d}$ with $d\mid \tfrac{q\pm 1}{o}$. If $d=p^{i}$, then $p\nmid \tfrac{q\pm 1}{o}$. Hence no cyclic subgroup of order $p^{i}$ arises from \eqref{eq:cyclic}, and this case can be discarded.

Henceforth assume that $|G|$ has a prime divisor $p'$ with $p'\ne p$ and $p'\ne 2$. Our next goal is to determine the pronormality of $\mathbb{Z}_{d}$ for $d\mid \tfrac{q\pm 1}{o}$ with $d$ not a power of $2$.

\begin{prop}\label{prop:2部分群を除いた巡回p'部分群は準正規}
Let $q:=p^{n}$. Suppose $|G|$ has a prime factor $p'$ different from $p$ and $2$, and set $j:=v_{p'}(|G|)$. Then $\mathbb{Z}_{p'^{i}}$ with $1\le i\le j$ is a pronormal subgroup in $G$. 
\end{prop}
\begin{proof}
Let $q:=p^{n}$. If $p'\neq 2$ and $p'\neq p$, then by Lemma \ref{lem:PSLのシローp部分群とその正規化群}, the Sylow $p'$-subgroup of $G$ is cyclic. In this case, pronormality follows immediately from Corollary \ref{cor:巡回シローp部分群の部分群も準正規}.
\end{proof}

\begin{prop}\label{prop:PSLで、奇数位数巡回部分群は準正規}
Let $q=p^{n}$ and $o:=\gcd(q-1,2)$. 
Assume $d\ge 1$ is odd and $d\mid \tfrac{q\pm 1}{o}$. 
Then \emph{every} subgroup of $G$ isomorphic to $\mathbb{Z}_{d}$ is pronormal in $G$.
\end{prop}

\begin{proof}
Fix a cyclic subgroup $H_{d}\subset G$ with $H_{d}\simeq \mathbb{Z}_{d}$. 
Since $d$ is odd and $d\mid \tfrac{q\pm 1}{o}$, all prime divisors of $d$ are different from $p$ and from $2$. 
Write $d=\prod_{i} p_{i}^{\,a_{i}}$ with $p_{i}\neq 2,p$ and let $P_{i}\subset H_{d}$ be the unique subgroup of order $p_{i}^{\,a_{i}}$, then $H_{d}=\prod_{i}P_{i}$ and the factors commute because $H_{d}$ is cyclic.

By Proposition~\ref{prop:2部分群を除いた巡回p'部分群は準正規}, each $P_{i}$ is pronormal in $G$. 
Since the $P_{i}$ commute, Lemma~\ref{lem:直積しても準正規性は保存} implies that their product $H_{d}$ is pronormal in $G$.

Finally, by Corollary~\ref{cor:PSL,J1,Szの巡回部分群の準正規性は代表1つを見ればok}, every subgroup of $G$ isomorphic to $\mathbb{Z}_{d}$ is pronormal.
\end{proof}

\begin{fac}[\cite{wilson2009}]\label{fac:巡回部分群の中心化群も巡回部分群}
Let $q=p^{n}$ and $o:=\gcd(q-1,2)$. 
If $A\subset G$ is cyclic of order $|A|=d$ with $d\mid \frac{q\pm1}{o}$ and $d\ne 2$, then the centralizer $C_G(A)$ is cyclic.
\end{fac}

\begin{prop}\label{prop:2べき以外の巡回部分群は準正規}
Let $q=p^{n}$. 
Fix $\ell\in \mathbb{Z}_{\ge 0}$ and $m\in \mathbb{Z}_{\ge 3}$ with $m$ odd, and let $H\subset G$ be a cyclic subgroup with $H\simeq \mathbb{Z}_{2^{\ell}m}$. 
Then every subgroup of $G$ isomorphic to $H$ is pronormal in $G$.
\end{prop}

\begin{proof}
Let $L\subset H$ be the unique subgroup of order $m$, so $L\simeq \mathbb{Z}_{m}$. 
Since $L\subset H$ and $H$ is abelian, we have $H\subset C_{G}(L)$. Hence, for every $g\in G$ we have $H^{g}\subset C_{G}(L^{g})$.

The subgroup $L$ has odd order and is cyclic, so by Proposition~\ref{prop:PSLで、奇数位数巡回部分群は準正規} it is pronormal in $G$. 
Therefore there exists $x\in \langle L,L^{g}\rangle\subset \langle H,H^{g}\rangle$ such that $L^{g}=L^{x}$, and consequently
$
H^{x}\subset C_{G}(L^{x})=C_{G}(L^{g}).
$
Applying Fact~\ref{fac:巡回部分群の中心化群も巡回部分群} to $A=L^{g}$ shows that $C_{G}(L^{g})$ is cyclic. 
In a cyclic group there is a unique subgroup of each order, so the two cyclic subgroups $H^{g}$ and $H^{x}$ of the same order $|H|$, both contained in $C_{G}(L^{g})$, must coincide: $H^{g}=H^{x}$. 
Thus $H$ is pronormal in $G$.

Finally, by Corollary~\ref{cor:PSL,J1,Szの巡回部分群の準正規性は代表1つを見ればok}, every subgroup of $G$ isomorphic to $H$ is pronormal.
\end{proof}

Assume $q=p^{n}$. The cases of $(\mathbb{Z}_{p})^{j}$ with $1\le j<n$ and of $\mathbb{Z}_{d}$ with $d\mid \tfrac{q\pm 1}{o}$ and $d$ not a power of $2$ have been settled. We now analyze the $2$-power subgroups $\mathbb{Z}_{2^{i}}$ and $D_{2^{j}}$ with $2^{i},2^{j}\mid \tfrac{q\pm 1}{o}$.

If \(q=2^{n}\), then \(2\nmid(q\pm1)\). Hence every \(d\) with \(d\mid \tfrac{q\pm1}{o}\) is odd. Consequently the families in \eqref{eq:dihedral} and \eqref{eq:cyclic} contain no cyclic or dihedral subgroups of \(2\)-power order. We therefore assume that \(q\) is odd.

\begin{prop}\label{prop:PSLの非準正規な2部分群}
Let $q$ odd, and put $k:=v_{2}(|G|)$. Then:
\[
\begin{aligned}
\textbf{(A)}\ & q\equiv \pm 1 \pmod 8:\ \ \text{all subgroups isomorphic to }D_{2^{j}}\ \text{with }1\le j\le k-2\ \text{are non\mbox{-}pronormal},\\
\textbf{(B)}\ & q\equiv \pm 3 \pmod 8:\ \ \text{all subgroups isomorphic to }\mathbb{Z}_{2}\ \text{are non\mbox{-}pronormal}.
\end{aligned}
\]
All other $2$-subgroups are pronormal. Note that $D_{2}\simeq \mathbb{Z}_{2}$ and $D_{4}\simeq (\mathbb{Z}_{2})^{2}$.
\end{prop}

\begin{proof}
\textbf{Case (A).}
By Lemma~\ref{lem:PSLのシローp部分群とその正規化群} a Sylow $2$-subgroup is $S\simeq D_{2^{k}}$ and $N_{G}(S)=S$. By Lemma~\ref{lem:二面体群の部分群の正規性} every $2$-subgroup of $S$ is cyclic or dihedral.

\emph{(A-1) Subgroups isomorphic to $\mathbb{Z}_{2}$.}
Fix $S$ and choose $P\subset S$ with $P\simeq \mathbb{Z}_{2}$ and $P\ntriangleleft S$ (Lemma~\ref{lem:二面体群の部分群の正規性}). Since $N_{G}(S)=S$, Lemma~\ref{lem:JohnRoseの補題} implies that $P$ is not pronormal in $G$. 
By Corollary~\ref{cor:PSL,J1,Szの巡回部分群の準正規性は代表1つを見ればok}, all subgroups isomorphic to $\mathbb{Z}_{2}$ are non\mbox{-}pronormal.

\emph{(A-2) Cyclic subgroups of order $2^{a}$ with $a\ge 2$.}
Let $P\subset G$ with $P\simeq \mathbb{Z}_{2^{a}}$ and take any Sylow $2$-subgroup $S'$ with $P\subset S'$. Then $S'\simeq D_{2^{k}}$ and $N_{G}(S')=S'$ (Lemma~\ref{lem:PSLのシローp部分群とその正規化群}). By Lemma~\ref{lem:二面体群の部分群の正規性}, every cyclic subgroup of order $>2$ is normal in $D_{2^{k}}$. Hence $P\triangleleft S'=N_{G}(S')$, and Lemma~\ref{lem:JohnRoseの補題} yields that $P$ is pronormal in $G$. 
By Corollary~\ref{cor:PSL,J1,Szの巡回部分群の準正規性は代表1つを見ればok}, all subgroups isomorphic to $\mathbb{Z}_{2^{a}}$ are pronormal.

\emph{(A-3) Dihedral subgroups of index at most $2$.}
Let $P\subset G$ with $P\simeq D_{2^{j}}$ and $j\in\{k,k-1\}$. For any Sylow $2$-subgroup $S'$ with $P\subset S'$ we have $S'\simeq D_{2^{k}}$ and $N_{G}(S')=S'$. Moreover, $D_{2^{k}}$ and $D_{2^{k-1}}$ are normal in $D_{2^{k}}$ (Lemma~\ref{lem:二面体群の部分群の正規性}). Thus $P\triangleleft S'=N_{G}(S')$, and Lemma~\ref{lem:JohnRoseの補題} gives that $P$ is pronormal in $G$. The conclusion holds for every subgroup isomorphic to $D_{2^{k}}$ or $D_{2^{k-1}}$ by the same argument.

\emph{(A-4) Dihedral subgroups of index at least $4$.}
Let $P\subset G$ with $P\simeq D_{2^{j}}$ and $2\le j\le k-2$. For any Sylow $2$-subgroup $S'$ with $P\subset S'$, one has $S'\simeq D_{2^{k}}$ and $N_{G}(S')=S'$, while no such $P$ is normal in $D_{2^{k}}$ (Lemma~\ref{lem:二面体群の部分群の正規性}). Hence $P\ntriangleleft S'=N_{G}(S')$, and Lemma~\ref{lem:JohnRoseの補題} shows that $P$ is not pronormal in $G$. The same reasoning applies to every subgroup isomorphic to $D_{2^{j}}$ with $2\le j\le k-2$.

\smallskip
\textbf{Case (B).}
By Lemma~\ref{lem:PSLのシローp部分群とその正規化群} a Sylow $2$-subgroup is $S\simeq (\mathbb{Z}_{2})^{2}$ and $N_{G}(S)\simeq A_{4}$. 
Fix $S$ and choose $P\subset S$ with $P\simeq \mathbb{Z}_{2}$. In $A_{4}$, no subgroup of order $2$ is normal, so $P\ntriangleleft N_{G}(S)$; hence, by Lemma~\ref{lem:JohnRoseの補題}, $P$ is not pronormal in $G$. 
By Corollary~\ref{cor:PSL,J1,Szの巡回部分群の準正規性は代表1つを見ればok}, all subgroups isomorphic to $\mathbb{Z}_{2}$ are non\mbox{-}pronormal. 
Finally, $S\triangleleft N_{G}(S)$ implies that $S$ is pronormal in $G$ by Lemma~\ref{lem:JohnRoseの補題}.
\end{proof}


In view of the preceding discussion and the reduction given by Proposition~\ref{prop:PSLで準正規性を求めるべき部分群}, the complete classification of pronormal subgroups of \(G\) is stated below.
The sets in \eqref{eq:PHとなるq全体} and \eqref{eq:NPrとなるq全体} determine the case division for the parameter \(q\).
We follow this division in what follows.

\begin{cor}\label{cor:PSLの準正規部分群分類}
Let \(q=p^{n}\) and \(k=v_{2}(|G|)\).
In each case below, every subgroup that is isomorphic to one of the listed groups is non\mbox{-}pronormal, and every other subgroup of \(G\) is pronormal.

\begin{description}[labelwidth=1.7cm]
\item[\textup{(PH–\(\bm{\mathcal{Q}_{2}}\))}] Assume \(q\in\mathcal{Q}_{2}\).
 Then \((\mathbb{Z}_{2})^{\,j}\) with \(1\le j<n\).

\item[\textup{(PH–\(\bm{\mathcal{Q}_{3}}\))}] Assume \(q\in\mathcal{Q}_{3}\).
 Then \(\mathbb{Z}_{2}\) and \((\mathbb{Z}_{3})^{\,j}\) with \(1\le j<n\).

\item[\textup{(PH–\(\bm{\mathcal{P}_{\pm3}}\))}] Assume \(q\in\mathcal{P}_{\pm3}\).
 Then \(\mathbb{Z}_{2}\).

\item[\textup{(PH–\(\bm{\mathcal{E}}\))}] Assume \(q\in\mathcal{E}\).
 For \(q=7\) one obtains \(\mathbb{Z}_{2}\).
 For \(q=17\) one obtains \(\mathbb{Z}_{2}\) and \((\mathbb{Z}_{2})^{2}\).

\item[\textup{(NPr)}] Assume \(q\in\mathcal{P}_{\pm1}^{(2)}\).
 Then \(D_{2^{j}}\) with \(1\le j\le k-2\).
\end{description}
\end{cor}

\begin{proof}
By Proposition~\ref{prop:PSLで準正規性を求めるべき部分群} the verification reduces to the case division in \eqref{eq:PHとなるq全体} and \eqref{eq:NPrとなるq全体}.
Throughout the proof, set \(q=p^{n}\).

We record two global facts that hold for every \(q\).
First, by Proposition~\ref{prop:2べき以外の巡回部分群は準正規}, every cyclic subgroup \(\mathbb{Z}_{d}\) with \(d\mid \tfrac{q\pm1}{o}\) and \(d\) not a power of \(2\) is \emph{pronormal}.
Second, by Proposition~\ref{prop:小さな初等p部分群は非準正規}, every elementary abelian \(p\)\nobreakdash-subgroup \((\mathbb{Z}_{p})^{\,j}\) with \(1\le j<n\) is \emph{non\mbox{-}pronormal}.
Consequently the remaining case analysis concerns only \(2\)\nobreakdash-subgroups, while the status of odd order cyclic subgroups and of \((\mathbb{Z}_{p})^{\,j}\) is already settled by these two facts.

\smallskip
\textbf{\(\bm{2}\)-subgroups in the \(\bm{\PH}\) regime.}
If \(q\in\mathcal{Q}_{2}\), then the non\mbox{-}pronormal subgroups are exactly \((\mathbb{Z}_{2})^{\,j}\) with \(1\le j<n\).
If \(q\in\mathcal{Q}_{3}\), then \(q\equiv 3\pmod 8\) holds because \(n\) is odd, so the \(2\)\nobreakdash-subgroup behavior agrees with the \(\mathcal{P}_{\pm3}\) case and Proposition~\ref{prop:PSLの非準正規な2部分群} yields that \(\mathbb{Z}_{2}\) is the unique non\mbox{-}pronormal \(2\)\nobreakdash-subgroup.
Together with the global fact on \((\mathbb{Z}_{p})^{\,j}\), this gives the item \textup{(PH–\(\mathcal{Q}_{3}\))}.
If \(q\in\mathcal{P}_{\pm3}\), then \(q=p\) is a prime, so \(n=1\).
Proposition~\ref{prop:PSLの非準正規な2部分群} gives \(\mathbb{Z}_{2}\) as the unique non\mbox{-}pronormal \(2\)\nobreakdash-subgroup.
There is no subgroup of the form \((\mathbb{Z}_{p})^{\,j}\) with \(1\le j<n\) in this case, and this is why such a family does not appear in \textup{(PH–\(\mathcal{P}_{\pm3}\))}.
If \(q\in\mathcal{E}\), then \(|\PSL(2,7)|=168\) so \(v_{2}(|G|)=3\) and \(|\PSL(2,17)|=2448\) so \(v_{2}(|G|)=4\).
A Sylow \(2\)\nobreakdash-subgroup is \(D_{2^{k}}\) with \(k=v_{2}(|G|)\).
Proposition~\ref{prop:PSLの非準正規な2部分群} implies that the non\mbox{-}pronormal \(2\)\nobreakdash-subgroups are \(D_{2^{j}}\) with \(1\le j\le k-2\).
This yields \(\mathbb{Z}_{2}\) for \(q=7\) and \(\mathbb{Z}_{2}\) together with \((\mathbb{Z}_{2})^{2}\) for \(q=17\) in (PH-$\mathcal{E}$).

\textbf{\(\bm{2}\)\nobreakdash-subgroups in the \(\bm{\NPr}\) regime.}
Assume \(q\in\mathcal{P}_{\pm1}^{(2)}\).
Then \(q=p\) is a prime and therefore \(n=1\).
Hence there is no subgroup of the form \((\mathbb{Z}_{p})^{\,j}\) with \(1\le j<n\).
By Proposition~\ref{prop:PSLの非準正規な2部分群} the non\mbox{-}pronormal \(2\)\nobreakdash-subgroups are precisely the dihedral groups \(D_{2^{j}}\) with \(1\le j\le k-2\).
\end{proof}

\subsubsection{Meet}
In this subsection we prove that \(\PrN(G)\) is not closed under meet for every value of \(q\) that appears in \eqref{eq:PHとなるq全体} and \eqref{eq:NPrとなるq全体}.

\begin{prop}\label{prop:PSLの準正規はmeetで閉じない}
For all \(q\) under consideration the family \(\PrN(G)\) is not closed under meet.
\end{prop}

\begin{proof}
We first verify that every subgroup isomorphic to \(D_{2(q\pm1)/o}\) is pronormal.

PH regime. By Theorem~\ref{thm:非冪零ならば準正規な単純群分類} every nonabelian subgroup of \(G\) is pronormal. The subgroups \(D_{2(q\pm1)/o}\) are nonabelian. Hence each subgroup isomorphic to \(D_{2(q\pm1)/o}\) is pronormal.

NPr regime. Here $q\in\mathcal{P}_{\pm1}^{(2)}$ is an odd prime and let $o=\gcd(q-1,2)$. Then
\[
v_{2}\!\left(\left|D_{2\cdot\frac{q\pm1}{o}}\right|\right)=v_{2}(q\pm1)\in\{1,\,v_{2}(|G|)\}.
\]
If $v_{2}(q\pm1)=1$, the dihedral group $D_{2\cdot\frac{q\pm1}{o}}$ is non-nilpotent, hence pronormal because $G$ is an NPr group. 
If $v_{2}(q\pm1)=v_{2}(|G|)$, then $D_{2\cdot\frac{q\pm1}{o}}$ is a Sylow $2$-subgroup of $G$, hence pronormal. 
Thus every subgroup isomorphic to $D_{2(q\pm1)/o}$ is pronormal.

We construct the meet. By Theorem~\ref{thm:dicsonのPSL部分群分類定理} the group \(G\) contains dihedral subgroups of the form \(D_{2(q+1)/o}\) and \(D_{2(q-1)/o}\) with \(o=\gcd(q-1,2)\). Choose
$
A\simeq \mathbb{Z}_{(q-1)/o},
X\simeq \mathbb{Z}_{(q+1)/o},
B,Y\simeq \mathbb{Z}_{2},
$
and set
$
H:=A\rtimes B  \simeq D_{2(q-1)/o},
K:=X\rtimes Y  \simeq D_{2(q+1)/o}.
$
By Lemma~\ref{lem:PSL,J1,SzはCSC群} all subgroups of order \(2\) are conjugate in \(G\). There exists \(g\in G\) with \(Y^{g}=B\). Then \(K^{g}=X^{g}\rtimes B\).

We claim that \(H\cap K^{g}=B\). One has
$
\gcd\!\left(|A|, |X^{g}|\right)=\gcd\!\left(\tfrac{q-1}{o},\,\tfrac{q+1}{o}\right)=1,
$
hence \(A\cap X^{g}=\{\id\}\). Both \(H\) and \(K^{g}\) contain the same involution \(B\). Therefore
$
H\cap K^{g}=B\ \simeq\ \mathbb{Z}_{2}.
$

By Corollary~\ref{cor:PSLの準正規部分群分類} each of \(H\) and \(K^{g}\) is pronormal for every \(q\) that appears in \eqref{eq:PHとなるq全体} and \eqref{eq:NPrとなるq全体}. The subgroup \(\mathbb{Z}_{2}\) is nonpronormal for every such \(q\) by the same corollary. Hence the meet of the two pronormal subgroups \(H\) and \(K^{g}\) equals the nonpronormal subgroup \(B\). This proves that \(\PrN(G)\) is not closed under meet.
\end{proof}

\begin{prop}\label{prop:PSLの準正規はjoinで閉じる}
For all \(q=p^{n}\) under consideration, \(\PrN(G)\) is closed under join.
\end{prop}

\begin{proof}
We use Corollary~\ref{cor:PSLの準正規部分群分類}.  
Every cyclic subgroup \(\mathbb{Z}_{d}\) of odd order with \(d\mid \tfrac{q\pm1}{o}\) is pronormal by Proposition~\ref{prop:2べき以外の巡回部分群は準正規}.  
Every elementary abelian \(p\)\nobreakdash-subgroup \((\mathbb{Z}_{p})^{\,j}\) with \(1\le j<n\) is non\mbox{-}pronormal by Proposition~\ref{prop:小さな初等p部分群は非準正規}.  
We show that any join of pronormal subgroups never equals a subgroup from the non\mbox{-}pronormal list.

\smallskip
\textup{\textbf{(PH–\(\bm{\mathcal{Q}_{2}}\)).}}  
The only non\mbox{-}pronormal subgroups are \((\mathbb{Z}_{2})^{\,j}\) with \(1\le j<n\).  
If \(\langle H,K\rangle=(\mathbb{Z}_{2})^{\,j}\), then \(H\subset(\mathbb{Z}_{2})^{\,j}\) and \(K\subset(\mathbb{Z}_{2})^{\,j}\).  
Both \(H\) and \(K\) are then non\mbox{-}pronormal by Proposition~\ref{prop:小さな初等p部分群は非準正規}, which contradicts that \(H\) and \(K\) are pronormal.  
Thus the join is pronormal.

\textup{\textbf{(PH–\(\bm{\mathcal{Q}_{3}}\)).}}   
The non\mbox{-}pronormal subgroups are \(\mathbb{Z}_{2}\) and \((\mathbb{Z}_{3})^{\,j}\) with \(1\le j<n\).  
If \(\langle H,K\rangle=\mathbb{Z}_{2}\), then \(H\) and \(K\) are both subgroups of order \(2\), hence both are non\mbox{-}pronormal, a contradiction.  
If \(\langle H,K\rangle=(\mathbb{Z}_{3})^{\,j}\), then \(H\) and \(K\) are both \(3\)\nobreakdash-groups contained in \((\mathbb{Z}_{3})^{\,j}\) and again both are non\mbox{-}pronormal by Proposition~\ref{prop:小さな初等p部分群は非準正規}.  
Thus the join is pronormal.

\textup{\textbf{(PH–\(\bm{\mathcal{P}_{\pm3}}\)).}}  
Here \(q=p\) so \(n=1\).  
The only non\mbox{-}pronormal subgroup is \(\mathbb{Z}_{2}\).  
If \(\langle H,K\rangle=\mathbb{Z}_{2}\), then both factors are subgroups of order \(2\), hence both are non\mbox{-}pronormal, a contradiction.  
Thus the join is pronormal.

\textup{\textbf{(PH–\(\bm{\mathcal{E}}\)).}}  
For \(q=7\) the only non\mbox{-}pronormal subgroup is \(\mathbb{Z}_{2}\).  
If \(\langle H,K\rangle=\mathbb{Z}_{2}\), then both \(H\) and \(K\) are subgroups of order \(2\), which is impossible since \(\mathbb{Z}_{2}\) is non\mbox{-}pronormal.  
For \(q=17\) the non\mbox{-}pronormal subgroups are \(\mathbb{Z}_{2}\) and \((\mathbb{Z}_{2})^{2}\).  
If \(\langle H,K\rangle\) equals one of these, then both \(H\) and \(K\) lie inside a non\mbox{-}pronormal elementary abelian \(2\)\nobreakdash-group, hence both are non\mbox{-}pronormal, a contradiction.  
Thus the join is pronormal.

\smallskip
\textup{\textbf{(\(\bm{\NPr}\)).}}  
Set \(k=v_{2}(|G|)\).  
By Corollary~\ref{cor:PSLの準正規部分群分類} the non\mbox{-}pronormal subgroups are exactly the dihedral groups \(D_{2^{j}}\) with \(1\le j\le k-2\).  
Let \(H\) and \(K\) be pronormal.  
If one factor is not a \(2\)\nobreakdash-subgroup then the join is not a \(2\)\nobreakdash-group and cannot be \(D_{2^{j}}\).  
Assume both \(H\) and \(K\) are \(2\)\nobreakdash-subgroups.  
By Proposition~\ref{prop:PSLの非準正規な2部分群} each of \(H\) and \(K\) is either \(\mathbb{Z}_{2^{a}}\) with \(a\ge 2\) or one of \(D_{2^{\,k-1}}\) and \(D_{2^{\,k}}\).  
If one factor contains \(D_{2^{\,k-1}}\) or \(D_{2^{\,k}}\), then the join has order at least \(2^{\,k-1}\) and is not \(D_{2^{j}}\) with \(j\le k-2\).  
If both factors are cyclic \(\mathbb{Z}_{2^{a}}\) and \(\mathbb{Z}_{2^{b}}\) with \(a,b\ge 2\), then the join is not dihedral since all elements of order greater than \(2\) lie in the rotation subgroup in a dihedral group of order at least \(8\) and this forces the join to be cyclic.  
Hence the join is not one of the non\mbox{-}pronormal dihedral groups.  
Therefore the join is pronormal.

We have shown that the join of pronormal subgroups never lands in the non\mbox{-}pronormal list in any regime.  
Hence \(\PrN(G)\) is closed under join.
\end{proof}

\section{Pronormal Subgroups and Lattice Structure of $\bm{J_{1}}$}

In this section, we discuss the pronormal subgroups and their lattice structure for $J_{1}$.

\subsection{Preliminaries for the Classification of Pronormal Subgroups in $\bm{J_{1}}$}

\begin{defi}\label{defi:J1におけるnotation}
We use the following notation throughout this section.
{\allowdisplaybreaks
\begin{flalign}
& \btri G := J_{1}: \text{ the first Janko group,}&& \\
& \btri \PrN(G) := \{\,H \subset G \mid H \text{ is pronormal in } G\,\}.
\end{flalign}
}

\end{defi}

For the subgroups of $J_1$, the following Table \ref{table:J1の部分群と準正規性} is known.

{\small
\begin{longtable}{|c|c|c|c||c|c|c|c|}
\caption{Subgroups and Details of $J_1$}
\label{table:J1の部分群と準正規性}\\
\hline
\multicolumn{8}{|c|}{\textbf{Subgroups and Details of $\bm{J_1}$}} \\
\hline
\textbf{Order} & \textbf{Structure} & \textbf{Abelian} & \textbf{Details} & \textbf{Order} & \textbf{Structure} & \textbf{Abelian} & \textbf{Details} \\
\hline
\endhead
$1$ & $1$ & Yes & & $22$ & $D_{22}$ & No & \\
\hline
\textbf{2} & \textbf{$\bm{\mathbb{Z}_2}$} & \textbf{Yes} &  & $24$ & $\mathbb{Z}_2 \times A_4$ & No & \\
\hline
$3$ & $\mathbb{Z}_3$ & Yes & (Sylow) & $30$ & $\mathbb{Z}_3 \times D_{10}$ & No & \\
\hline
\textbf{4} & $\bm{(\mathbb{Z}_{2})^{2}}$ & \textbf{Yes} & & $38$ & $D_{38}$ & No & \\
\hline
$5$ & $\mathbb{Z}_5$ & Yes & (Sylow) & $42$ & $\mathbb{Z}_7 : \mathbb{Z}_6$ & No & \\
\hline
$6$ & $S_3$ & No & & $55$ & $\mathbb{Z}_{11} : \mathbb{Z}_5$ & No & \\
\hline
$7$ & $\mathbb{Z}_7$ & Yes & (Sylow) & $56$ & $(\mathbb{Z}_2)^{3} : \mathbb{Z}_7\simeq \AGL(1,8)$ & No & \\
\hline
$8$ & $(\mathbb{Z}_2)^{3}$ & Yes & (Sylow) & $57$ & $\mathbb{Z}_{19} : \mathbb{Z}_3$ & No & \\
\hline
$10$ & $D_{10}$ & No & & $60$ & $A_5$ & No & \\
\hline
$11$ & $\mathbb{Z}_{11}$ & Yes & (Sylow) & $110$ & $\mathbb{Z}_{11} : \mathbb{Z}_{10}$ & No & \\
\hline
$12$ & $A_4$ & No & & $114$ & $\mathbb{Z}_{19} : \mathbb{Z}_6$ & No & \\
\hline
$14$ & $D_{14}$ & No & & $120$ & $\mathbb{Z}_2 \times A_5$ & No & \\
\hline
$15$ & $\mathbb{Z}_{15}$ & Yes & (Hall-$\{3,5\}$) & $168$ & $(\mathbb{Z}_{2})^{3}:(\mathbb{Z}_{7}:\mathbb{Z}_{3})\simeq \AGaL(1,8)$ & No & \\
\hline
$19$ & $\mathbb{Z}_{19}$ & Yes & (Sylow) & $660$ & $\mathrm{PSL}(2,11)$ & No & \\
\hline
$20$ & $D_{20}$ & No & & $175560$ & $J_1$ & No & \\
\hline
$21$ & $\mathbb{Z}_7 : \mathbb{Z}_3$ & No & & & & & \\
\hline
\end{longtable}
}
\normalsize

\subsection{Classification of Pronormal Subgroups of $\bm{J_1}$}
\begin{prop}\label{lem:C2,C22以外はJ1の準正規部分群}
Every subgroup \(H\subset G\) with 
\(H\not\simeq \mathbb{Z}_2, (\mathbb{Z}_2)^2\) is pronormal.
\end{prop}

\begin{proof}
By Theorem \ref{thm:プロハミルトン群分類}, $G$ is a prohamiltonian group, so every non-abelian subgroup is pronormal. The order of $G$ is $|G| = 175560 = 2^3 \cdot 3 \cdot 5 \cdot 7 \cdot 11 \cdot 19$. By Lemma \ref{lem:準正規部分群のクラス}, the Sylow $p$-subgroups and Hall $\pi$-subgroups of $G$ are pronormal. From Table \ref{table:J1の部分群と準正規性}, among the abelian subgroups, if we exclude these Sylow and Hall subgroups, only $\mathbb{Z}_2$ and $(\mathbb{Z}_2)^2$ remain. Therefore, all subgroups except these two $2$-subgroups are pronormal.
\end{proof}

By Theorem \ref{thm:プロハミルトン群分類} and Table \ref{table:J1の部分群と準正規性}, among the subgroups of $G$, the parts whose pronormality is unknown are $\mathbb{Z}_{2}$ and $(\mathbb{Z}_{2})^{2}$. Our goal is to clarify the pronormality of these subgroups.

\begin{remark}[\cite{DixonMortimer1996,Janko1966}]\label{rem:J1のシロー2部分群の正規化群とその正規部分群}

The Sylow $2$-subgroup of $G$ is $S:=(\mathbb{Z}_{2})^{3}$, and $N_{G}(S)=(\mathbb{Z}_{2})^{3}:(\mathbb{Z}_{7}:\mathbb{Z}_{3})\simeq \AGaL(1,8)$. The four subgroups $\{\id\}$, $(\mathbb{Z}_{2})^{3}$, $\AGL(1,8)$, and $\AGaL(1,8)$ are the normal subgroups of $\AGaL(1,8)$, and all are characteristic. All other subgroups are non-normal.
\end{remark}

\begin{cor}\label{cor:J1の準正規部分群の分類}
All subgroups of $G$ are pronormal except $\mathbb{Z}_{2}$ and $(\mathbb{Z}_{2})^{2}$.
\end{cor}

\begin{proof}
Let $S$ be a Sylow $2$-subgroup of $G$. Then $S\simeq (\mathbb{Z}_{2})^{3}$. Let $P\subset S$ with $P\simeq \mathbb{Z}_{2}$. By Lemma~\ref{lem:JohnRoseの補題}, $P$ is pronormal in $G$ if and only if $P\triangleleft N_{G}(S)$. 

By Remark~\ref{rem:J1のシロー2部分群の正規化群とその正規部分群}, no subgroup $P\simeq \mathbb{Z}_{2}$ is normal in $N_{G}(S)$. Hence, by the above equivalence, $P$ is not pronormal in $G$. The case $P\simeq (\mathbb{Z}_{2})^{2}$ is analogous. Under the same hypothesis, $P$ is not normal in $N_{G}(S)$ and therefore $P$ is not pronormal in $G$. The remaining subgroups are pronormal by Proposition~\ref{lem:C2,C22以外はJ1の準正規部分群}.
\end{proof}

\subsection{Investigation of Whether $\bm{\PrN(J_{1})}$ Form a Lattice}

\subsubsection{meet}

\begin{prop}\label{prop:J1の準正規はmeetで閉じない}
$\PrN(G)$ is not closed under meets.  
\end{prop}
\begin{proof}
Define $H:=A\rtimes B\simeq \mathbb{Z}_{3}\rtimes \mathbb{Z}_{2}\simeq D_{6}$ and $K:=X\rtimes Y\simeq \mathbb{Z}_{5}\rtimes \mathbb{Z}_{2}\simeq D_{10}$. Since $J_{1}$ is a csc-group by Lemma~\ref{lem:PSL,J1,SzはCSC群}, there exists $g\in G$ such that $K^{g}=X^{g}\rtimes B$. Then $H\cap K^{g}\simeq B\simeq \mathbb{Z}_{2}$. By Corollary \ref{cor:J1の準正規部分群の分類}, $D_{6}$ and $D_{10}$ are pronormal subgroups, but $\mathbb{Z}_{2}$ is non-pronormal. Therefore, $\PrN(G)$ is not closed under meets.
\end{proof}

\subsubsection{join}

\begin{prop}\label{prop:J1の準正規はjoinで閉じる}
$\PrN(G)$ is closed under joins.
\end{prop}
\begin{proof}
By Corollary \ref{cor:J1の準正規部分群の分類}, the non-pronormal subgroups in $J_1$ are $\mathbb{Z}_{2}$ and $(\mathbb{Z}_{2})^{2}$. When we consider the join of pronormal subgroups, as evident from Table \ref{table:J1の部分群と準正規性}, the order relations ensure that the result cannot be isomorphic to $\mathbb{Z}_{2}$ or $(\mathbb{Z}_{2})^{2}$.
\end{proof}

\section{Pronormal Subgroups and Lattice Structure of $\bm{\Sz(2^{2n+1})}$}

In this section, we discuss the pronormal subgroups and their lattice structure for $\Sz(q)$, where $q:=2^{2n+1}$ and $2n+1$ is prime.

\subsection{Preliminaries for the Classification of Pronormal Subgroups in $\bm{\Sz(q)}$}

\begin{defi}\label{defi:Szにおけるnotation}
We use the following notation throughout this section.
{\allowdisplaybreaks
\begin{flalign*}
& \btri\; q := 2^{2n+1}, \theta := 2^{n+1} \text{ with } n \in \mathbb{Z}_{\ge 1}, &&\\
& \btri\; m_{+} := q+\theta+1, m_{-} := q-\theta+1, &&\\
& \btri\; G := \Sz(q) \text{: the Suzuki group of Lie type}, &&\\
& \btri\; E_{q} \simeq (\mathbb{Z}_{2})^{2n+1} \text{: the elementary abelian $2$-group of order $q$}, &&\\
& \btri\; S \text{: a Sylow $2$-subgroup of $G$, isomorphic to } E_{q}.E_{q}, &&\\
& \btri\; F \text{: the normalizer of $S$ in $G$, isomorphic to } S \rtimes \mathbb{Z}_{q-1}, &&\\
& \btri\; Z(S) \text{: the center of $S$, isomorphic to } E_{q}, &&\\
& \btri\; \PrN(G) := \{\,H \subset G \mid H \text{ is pronormal in } G\,\}, &&\\
& \btri\; \alpha \text{: a primitive element of } \mathrm{GF}(q), &&\\
& \btri\; 
U := \begin{psmallmatrix}
1 & 0 & 0 & 0 \\
1 & 1 & 0 & 0 \\
1 & 1 & 1 & 0 \\
1 & 0 & 1 & 1
\end{psmallmatrix},\quad
T(\alpha) := \begin{psmallmatrix}
\alpha & 0 & 0 & 0 \\
0 & \alpha^{\theta-1} & 0 & 0 \\
0 & 0 & \alpha^{-\theta+1} & 0 \\
0 & 0 & 0 & \alpha^{-1}
\end{psmallmatrix},\quad
w := \begin{psmallmatrix}
0 & 0 & 0 & 1 \\
0 & 0 & 1 & 0 \\
0 & 1 & 0 & 0 \\
1 & 0 & 0 & 0
\end{psmallmatrix},\quad
T := \langle T(\alpha) \rangle, &&\\
& \btri\;
M(a,b) := \begin{psmallmatrix}
1 & 0 & 0 & 0 \\
a & 1 & 0 & 0 \\
a^{\theta+1}+b & a^{\theta} & 1 & 0 \\
- a^{\theta+2} + b^{\theta} + a b & b & a & 1
\end{psmallmatrix}. && 
\end{flalign*}
}
  
\end{defi}

Regardless of whether $2n+1$ is prime or not, it is known that the subgroups of $\Sz(q)$ are completely classified as follows.
\begin{fac}[{\cite{wilson2009}}]\label{fac:リー型鈴木群の部分群分類定理}
The subgroups of $G$ can be classified as follows.
{\allowdisplaybreaks
\begin{flalign}
& \btri\; \text{$\mathrm{\Sz}(q_{0})$, where $q_{0}^{\,k}=q$ with $k$ an odd prime and $q_{0}>2$.} && \label{eq:Szq0}\\
& \btri\; \text{Solvable Frobenius subgroups $F\simeq (E_{q}.E_{q})\rtimes \mathbb{Z}_{q-1}$ of order $q^{2}(q-1)$ and their subgroups.} && \label{eq:Frobenius}\\
& \btri\; \text{Dihedral subgroups $D_{2(q-1)}$ of order $2(q-1)$ and their subgroups.} && \label{eq:Dihedral}\\
& \btri\; \text{Frobenius semidirect subgroups $\mathbb{Z}_{m_{\pm}}\rtimes \mathbb{Z}_{4}$ of order $4m_{\pm}$ and their subgroups.} && \label{eq:Semidirect}
\end{flalign}
}
In particular, every subgroup of $G$ is contained in a subgroup of one of these types.
\end{fac}

\subsection{Classification of Pronormal Subgroups of $\bm{\Sz(2^{2n+1})}$}

\begin{prop}\label{prop:鈴木群においてp部分群と合成数位数dの巡回群以外は準正規}
Assume that \(2n+1\) is prime. Let \(d\) be a composite number that is not a prime power. Then every non-cyclic subgroup of \(G\) of order \(d\) is pronormal.
\end{prop}

\begin{proof}
By Fact \ref{fac:リー型鈴木群の部分群分類定理} every subgroup of \(G\) is contained in a subgroup of one of \eqref{eq:Szq0}, \eqref{eq:Frobenius}, \eqref{eq:Dihedral}, \eqref{eq:Semidirect}. By Theorem \ref{thm:非冪零ならば準正規な単純群分類}, under the hypothesis that \(2n+1\) is prime, it suffices to prove that each subgroup under consideration is non-nilpotent.

\emph{Case \eqref{eq:Frobenius}.} Let \(F\simeq (E_q.E_q)\rtimes \mathbb{Z}_{q-1}\) and let \(F'\subset F\) have order \(d\) and be non-cyclic. By Remark \ref{rem:フロベニウスの部分群はフロベニウス} there exist subgroups \(A\subset E_q.E_q\) and \(B\subset \mathbb{Z}_{q-1}\) with \(F'=A\rtimes B\). If \(A=\{\id\}\) then \(F'=B\) is cyclic and is excluded by the hypothesis. If \(B=\{\id\}\) then \(F'=A\) is a \(2\)-group and is excluded by the hypothesis. Thus \(A\) and \(B\) are both nontrivial, so \(F'=A\rtimes B\) with nontrivial action. In particular \(F'\) does not degenerate to a direct product. Here \(A\) is a \(2\)-group and \(B\) has odd order. A finite group is nilpotent if and only if it is the direct product of its Sylow subgroups, which does not occur for \(F'\). Hence \(F'\) is non-nilpotent. Therefore, by Theorem \ref{thm:非冪零ならば準正規な単純群分類}, every such \(F'\) is pronormal.

\emph{Case \eqref{eq:Dihedral}.} Let \(L=D_{2(q-1)}\) and let \(L'\subset L\) have order \(d\) and be non-cyclic. Every subgroup of \(L\) is either cyclic of odd order dividing \(q-1\) or dihedral \(D_{2s}\) of order \(2s\) with \(s\mid(q-1)\) and \(s\) odd. Cyclic subgroups are excluded by the hypothesis. A dihedral group is nilpotent if and only if it is a \(2\)-group, and since \(q-1\) is odd, the only \(2\)-group among subgroups of \(L\) is \(D_{2}\simeq \mathbb{Z}_{2}\), which is excluded by the hypothesis. Hence any non-cyclic \(L'\) is non-nilpotent. Therefore \(L'\) is pronormal by Theorem \ref{thm:非冪零ならば準正規な単純群分類}.

\emph{Case \eqref{eq:Semidirect}.} Let \(N=\mathbb{Z}_{m_{\pm}}\rtimes \mathbb{Z}_{4}\) and let \(N'\subset N\) have order \(d\) and be non-cyclic. Write \(A\subset \mathbb{Z}_{m_{\pm}}\) and \(B\subset \mathbb{Z}_{4}\) so that \(N'=A\rtimes B\) as in Remark \ref{rem:フロベニウスの部分群はフロベニウス}. If \(A=\{\id\}\), then \(N'=B\) is cyclic and is excluded. If \(B=\{\id\}\), then \(N'=A\) has odd order and is excluded. Thus \(A\) and \(B\) are both nontrivial and the action is nontrivial, so \(N'\) does not split as a direct product. Since \(A\) has odd order and \(B\) is a \(2\)-group, \(N'\) cannot be the direct product of its Sylow subgroups. Hence \(N'\) is non-nilpotent and therefore pronormal by Theorem \ref{thm:非冪零ならば準正規な単純群分類}.

\emph{Case \eqref{eq:Szq0}.} Let \(H\simeq \Sz(q_{0})\) with \(q_{0}^{\,k}=q\) and \(k\) an odd prime and \(q_{0}>2\). If such a subgroup occurs, then \(H\) is simple and therefore non-nilpotent. Moreover, by applying Fact \ref{fac:リー型鈴木群の部分群分類定理} inside \(H\), any non-cyclic subgroup of order \(d\) in \(H\) lies in a subgroup of one of \eqref{eq:Frobenius}, \eqref{eq:Dihedral}, \eqref{eq:Semidirect}, and the preceding cases show that it is non-nilpotent.

In all cases the subgroup under consideration is non-nilpotent. The conclusion follows from Theorem \ref{thm:非冪零ならば準正規な単純群分類}.
\end{proof}

Therefore, by Proposition \ref{prop:鈴木群においてp部分群と合成数位数dの巡回群以外は準正規}, when $2n+1$ is prime, it suffices to investigate the pronormality for the following two types of subgroups.

\begin{center}
(a): $2$-subgroups of $G$,\quad 
(b): cyclic groups $\mathbb{Z}_{d}$, where $d\mid m_{\pm}$ or $d\mid q-1$ \,($d$ odd).
\end{center}

\subsubsection{Discussion of Type (a)}

Let \(S\) be a Sylow \(2\)-subgroup of \(G\). Then \(S\simeq E_{q}.E_{q}\). Put \(N=N_{G}(S)=ST\simeq S\rtimes \mathbb{Z}_{q-1}\). The group \(N\) is a Frobenius group. We work in the standard matrix model of \(\Sz(q)\) where
\(G=\langle T(\alpha),U,w\rangle\),
\(S=\{M(a,b)\mid a,b\in \GF(q)\}\),
\(Z(S)=\{M(0,b)\mid b\in \GF(q)\}\),
and \(Z(S)\subset S\subset N\subset G\)
(see \cite{wilson2009}).

Let \(P\) be any \(2\)-subgroup of \(G\). By Lemma~\ref{lem:JohnRoseの補題} the subgroup \(P\) is pronormal in \(G\) if and only if for every Sylow \(2\)-subgroup \(S'\) of \(G\) with \(P\subset S'\) one has \(P\triangleleft N_{G}(S')\). If there exists a Sylow \(2\)-subgroup \(S'\) of \(G\) with \(P\subset S'\) and \(P\ntriangleleft N_{G}(S')\), then \(P\) is not pronormal in \(G\). We investigate \(2\)-subgroups \(P\) under this normality criterion.

\begin{lem}\label{lem:等式1}  
For any $\mu\in\GF(q)$, the equality $T(\alpha^{j})^{-1}M(0,\mu)T(\alpha^{j})=M(0,\mu (\alpha^{\theta})^{j})$ holds.
\end{lem}
\begin{proof}
Direct computation shows that both sides agree in all components, using the fact that $\alpha^{\theta^{2}}=\alpha^{2}$.
\end{proof}

\begin{lem}\label{lem:等式2}
Fix $i\in \{0,1,\dotsc,q-2\}$. Then the following equality holds:
\[\{T(\alpha^{j})^{-1}M(0,\alpha^{i})T(\alpha^{j})\mid j=0,1,\dotsc,q-2\} = \{M(0,\alpha^{j'})\mid j'=0,1,\dotsc,q-2\}=Z(S)\setminus \{\id\} .\]  
\end{lem}
\begin{proof}
Since $\gcd(\theta,q-1)=1$, we have that $\alpha^{\theta}$ is also a primitive element. Applying Lemma \ref{lem:等式1} with $\mu:= \alpha^{i}$, we obtain
\[T(\alpha^{j})^{-1}M(0,\alpha^{i})T(\alpha^{j})=M(0,\alpha^{i} (\alpha^{\theta})^{j}).\]
As $j$ ranges over $\{0,1,\dotsc,q-2\}$, the expression $\alpha^{i} (\alpha^{\theta})^{j}$ ranges over all elements of $\GF(q)^{\times}$, which gives the desired equality.
\end{proof}

\begin{prop}\label{prop:任意のSTの正規部分群は中心を包含する}
Every nontrivial normal $2$-subgroup of \(ST\) contains \(Z(S)\).
\end{prop}

\begin{proof}
Write \(N:=ST\), \(T\simeq \mathbb{Z}_{q-1}\), and
\[
Z(S)=\{\,M(0,b)\mid b\in \GF(q)\,\}=\{\id\}\cup\{\,M(0,\alpha^{j})\mid j=0,1,\dots,q-2\,\}.
\]
Let $L$ be a nontrivial normal subgroup of $ST$, and suppose $M(0,\alpha^{i})\in L$ for some $i$. Since $L\triangleleft ST$, conjugation by elements $T(\alpha^{j})\in S$ gives
\[\{T(\alpha^{j})^{-1}M(0,\alpha^{i})T(\alpha^{j})\mid j=0,1,\dotsc,q-2\}\subset L.\]
By Lemma \ref{lem:等式2}, this set equals $Z(S)$, hence $Z(S)\subset L$.
\end{proof}

\begin{lem}\label{lem:等式3}
For any $\tau\in \GF(q)$, the equality 
$T(\alpha^{j})^{-1}M(\tau ,0)T(\alpha^{j})=M(\alpha^{j(2-\theta)}\tau ,0)$ holds. In particular, $\alpha^{(2-\theta)}$ is a primitive element of $\GF(q)$.
\end{lem}
\begin{proof}
Direct computation using $\alpha^{\theta^{2}}=\alpha^{2}$ shows both sides are equal. Since $\gcd(q-1, 2-\theta)=1$, we have that $\alpha^{(2-\theta)}$ is primitive.
\end{proof}

\begin{prop}\label{prop:Lの正規部分群LがZより大きいならSを包含する}
Let $L$ be a normal $2$-subgroup of $ST$ with $Z(S)\subsetneq L$. Then $S\subset L$.
\end{prop}

\begin{proof}
Put $N=ST$ and $T\simeq \mathbb{Z}_{q-1}$. Since $|T|$ is odd and $S$ is a Sylow $2$-subgroup of $N$, we have $L\subset S$.

We have $Z(S)=\{M(0,b)\mid b\in \GF(q)\}$. The strict inclusion $Z(S)\subsetneq L$ implies the existence of an element $M(x,y)\in L$ with $x\ne 0$. Because $L\lhd N$, conjugation by every $T(\alpha^{j})\in T$ preserves $L$. By Lemmas~\ref{lem:等式1} and~\ref{lem:等式3} we obtain
\[
T(\alpha^{j})^{-1}M(0,y)T(\alpha^{j})=M\bigl(0,(\alpha^{\theta})^{j}y\bigr)\in L,
\qquad
T(\alpha^{j})^{-1}M(x,0)T(\alpha^{j})=M\bigl(\alpha^{j(2-\theta)}x,0\bigr)\in L,
\]
and therefore
\[
T(\alpha^{j})^{-1}M(x,y)T(\alpha^{j})
= M\bigl(\alpha^{j(2-\theta)}x,(\alpha^{\theta})^{j}y\bigr)\in L
\quad\text{for all } j\in\{0,1,\dots,q-2\}.
\]
By Lemma~\ref{lem:等式3}, $\alpha^{\,2-\theta}$ is primitive. Hence, as $j$ varies, the first coordinate $\alpha^{j(2-\theta)}x$ runs through all of $\GF(q)^{\times}$. In addition, $Z(S)\subset L$ gives $M(0,\mu)\in L$ for every $\mu\in \GF(q)$, and thus
\[
M\bigl(\alpha^{j(2-\theta)}x,(\alpha^{\theta})^{j}y\bigr)\,M(0,\mu)
= M\bigl(\alpha^{j(2-\theta)}x,(\alpha^{\theta})^{j}y+\mu\bigr)\in L.
\]
We conclude that all elements $M(u,v)$ with $u\in\GF(q)^{\times}$ and $v\in\GF(q)$ lie in $L$. Together with $Z(S)\subset L$ this yields $S\subset L$.
\end{proof}

\begin{prop}\label{prop:鈴木群の2部分群の準正規性判定完成}
A nontrivial $2$-subgroup of $G$ is pronormal if and only if it is isomorphic to $E_{q}$ or to $E_{q}.E_{q}$.
\end{prop}

\begin{proof}
Let \(S\) be a Sylow \(2\)-subgroup of \(G\) and write \(N:=N_{G}(S)=ST\). By Propositions~\ref{prop:任意のSTの正規部分群は中心を包含する} and \ref{prop:Lの正規部分群LがZより大きいならSを包含する} the nontrivial normal \(2\)-subgroups of \(N\) are exactly \(Z(S)\) and \(S\).

Let \(P\) be a nontrivial \(2\)-subgroup of \(G\). Choose a Sylow \(2\)-subgroup \(S'\) of \(G\) with \(P\subset S'\). By Lemma~\ref{lem:JohnRoseの補題} the subgroup \(P\) is pronormal in \(G\) if and only if \(P\triangleleft N_{G}(S')\). In particular, when \(S'=S\) we obtain \(P\triangleleft N_{G}(S)\). The description of normal \(2\)-subgroups of \(N_{G}(S)\) then forces \(P=Z(S)\) or \(P=S\).

Conversely \(Z(S)\triangleleft N_{G}(S)\) and \(S\triangleleft N_{G}(S)\). Hence Lemma~\ref{lem:JohnRoseの補題} gives that \(Z(S)\) and \(S\) are pronormal in \(G\). This completes the proof.
\end{proof}

\subsubsection{Discussion of Type (b)}

\begin{prop}\label{prop:Sz-奇数dの巡回部分群は準正規}
Let \(G=\Sz(q)\) where \(q=2^{2m+1}\).
Let \(d\ge 1\) be odd and assume that \(d\mid m_{\pm}\) or \(d\mid(q-1)\).
Then every subgroup of \(G\) that is isomorphic to \(\mathbb{Z}_{d}\) is pronormal in \(G\).
\end{prop}

\begin{proof}
Fix a cyclic subgroup \(H\subset G\) with \(|H|=d\).
Write the prime factorization \(d=\prod_{i}p_{i}^{\,a_{i}}\) with distinct odd primes \(p_{i}\).
Inside \(H\) there is a unique subgroup \(P_{i}\subset H\) of order \(p_{i}^{\,a_{i}}\) for each \(i\). 
Hence \(H=\prod_{i}P_{i}\) as an internal direct product.

For each \(i\), the prime \(p_{i}\) divides \(m_{\pm}\) or \(q-1\).
In \(\Sz(q)\) the Sylow \(p_{i}\)-subgroup is cyclic for every such odd prime \(p_{i}\).
By Corollary~\ref{cor:巡回シローp部分群の部分群も準正規}, every subgroup of a cyclic Sylow \(p_{i}\)-subgroup is pronormal in \(G\).
Therefore each \(P_{i}\) is pronormal in \(G\).

By Lemma~\ref{lem:直積しても準正規性は保存}, the product of commuting pronormal subgroups is again pronormal.
Applying this to the family \(\{P_{i}\}_{i}\) yields that \(H=\prod_{i}P_{i}\) is pronormal in \(G\).

By Corollary~\ref{cor:PSL,J1,Szの巡回部分群の準正規性は代表1つを見ればok}, once one representative \(H\simeq\mathbb{Z}_{d}\) is pronormal, every subgroup of \(G\) that is isomorphic to \(\mathbb{Z}_{d}\) is pronormal.
This proves the claim.
\end{proof}

Combining Propositions \ref{prop:鈴木群の2部分群の準正規性判定完成} and \ref{prop:Sz-奇数dの巡回部分群は準正規} we obtain the following corollary.

\begin{cor}\label{cor:鈴木群の準正規部分群分類定理}
Let \(G=\Sz(q)\) with \(q=2^{2n+1}\) and assume that \(2n+1\) is prime. Then every subgroup of \(G\) is pronormal unless it is a \(2\)-subgroup that is different from \(E_{q}\) and \(E_{q}.E_{q}\). Conversely, every \(2\)-subgroup other than \(E_{q}\) and \(E_{q}.E_{q}\) is not pronormal.
\end{cor}

\subsection{Investigation of Whether $\bm{\PrN(\Sz(2^{2n+1}))}$ Form a Lattice}

\subsubsection{meet}

By Lemma~\ref{lem:PSL,J1,SzはCSC群}, $G$ is a csc-group, hence all cyclic subgroups of the same order are conjugate.

\begin{prop}\label{prop:Szの準正規はmeetで閉じない}
$\PrN(G)$ is not closed under meets.  
\end{prop}

\begin{proof}
Let $H_{+}:=A\rtimes B\simeq \mathbb{Z}_{m_{+}}\rtimes \mathbb{Z}_{4}$ and $H_{-}:=X\rtimes Y\simeq \mathbb{Z}_{m_{-}}\rtimes \mathbb{Z}_{4}$. By Lemma~\ref{lem:PSL,J1,SzはCSC群}, $G$ is a csc-group, so there exists $g\in G$ such that $H_{-}^{g}=X^{g}\rtimes B$. Then $H_{+}\cap H_{-}^{g}= B\simeq \mathbb{Z}_{4}$. By Corollary \ref{cor:鈴木群の準正規部分群分類定理}, $\mathbb{Z}_{m_{\pm}}\rtimes \mathbb{Z}_{4}$ are pronormal subgroups, but $\mathbb{Z}_{4}$ is non-pronormal. Therefore $\PrN(G)$ is not closed under meets.
\end{proof}

\subsubsection{join}

\begin{lem}\label{lem:非冪零群のjoinもまた非冪零群}
Let $H,K\subset G$ be any non-nilpotent subgroups. Then $\langle H,K\rangle$ is also a non-nilpotent subgroup.
\end{lem}
\begin{proof}
Suppose for contradiction that $H,K\subset G$ are non-nilpotent subgroups but $\langle H,K \rangle$ is nilpotent. Since subgroups of nilpotent groups are also nilpotent, both $H$ and $K$ would be nilpotent, contradicting our assumption. Therefore, $\langle H,K \rangle$ is non-nilpotent.
\end{proof}

\begin{prop}\label{prop:Szの準正規はjoinで閉じている}
Assume $2n+1$ is prime. Then $\PrN(G)$ is closed under join.
\end{prop}
\begin{proof}
When $2n+1$ is prime, all non-nilpotent subgroups of $\Sz(q)$ are pronormal subgroups. By Lemma \ref{lem:非冪零群のjoinもまた非冪零群}, if $H$ and $K$ are non-nilpotent pronormal subgroups, then $\langle H,K \rangle$ is also non-nilpotent and hence pronormal. Therefore, for non-nilpotent pronormal subgroups, $\PrN(G)$ is always closed under join. By Corollary \ref{cor:鈴木群の準正規部分群分類定理}, the nilpotent pronormal subgroups are precisely:
\[(a): E_{q},\, E_{q}.E_{q}\quad (b): \mathbb{Z}_{d}\text{ where }d\mid m_{\pm}, q-1.  \]

Note that the non-pronormal subgroups in $G$ are exactly the $2$-subgroups other than those in (a). We now consider what happens to $\langle H,K\rangle$ for various combinations of pronormal subgroups $H,K\subset G$.

\textbf{Case (1):} When one of $H,K$ is $E_{q}.E_{q}$ and the other is arbitrary, we have $|\langle H,K \rangle|\geq q^{2}$. Such subgroups are pronormal unless they are $2$-subgroups, but any $2$-subgroup of order at least $q^{2}$ must be a Sylow $2$-subgroup, which is pronormal.

\textbf{Case (2):} When one of $H,K$ is $\mathbb{Z}_{d}$ with $d\mid m_{\pm}, q-1$, then $\langle H,K \rangle$ contains elements of odd order, so it is not a $2$-subgroup and hence is pronormal.

\textbf{Case (3):} When one of $H,K$ is $E_{q}$, if the other is a pronormal subgroup not isomorphic to $E_{q}$, then $\langle H,K \rangle$ has order at least $q^{2}$ or contains elements of odd order, making it pronormal. Thus we need to consider whether the join of two distinct subgroups isomorphic to $E_{q}$ is again pronormal.

Let $M:=\langle H,K \rangle$ where $H,K \simeq E_{q}$ with $H \neq K$. Note that $Z(S)\simeq H,K$ and $|H|=|K|=q$. By Fact \ref{fac:リー型鈴木群の部分群分類定理}, the proper maximal subgroups of $\Sz(q)$ are classified into four types. We prove by contradiction that $M = \Sz(q)$. Assume $M \subsetneq \Sz(q)$. Then $M$ must be contained in one of the four types of maximal subgroups.

\textbf{Case (3-i):}
Assume for a contradiction that $M\subset \Sz(q_{0})$ with $q=q_{0}^{k}$, where $k\ge 3$ is an odd prime.
Write $q_{0}=2^{\,2m+1}$ with $m\ge 1$. Then
\[
|\Sz(q_{0})| \;=\; q_{0}^{2}\,(q_{0}^{2}+1)\,(q_{0}-1),
\]
and both $q_{0}^{2}+1$ and $q_{0}-1$ are odd.
We have
\[
v_{2}\bigl(|\Sz(q_{0})|\bigr)
= v_{2}(q_{0}^{2})
= 2\,v_{2}(q_{0})
= 2(2m+1).
\]
Since $H\simeq E_{q}$ has order $|H|=q=q_{0}^{k}$, it follows that
\[
v_{2}(|H|) \;=\; v_{2}(q_{0}^{k}) \;=\; k\,v_{2}(q_{0})
\;=\; k(2m+1) \;>\; 2(2m+1)
\;=\; v_{2}\bigl(|\Sz(q_{0})|\bigr),
\]
because $k\ge 3$. Hence
$v_{2}(|H|)\le v_{2}(|\Sz(q_{0})|)$ must hold a contradiction.
Therefore $M\not\subset \Sz(q_{0})$.

\textbf{Case (3-ii).}
Fix a subgroup $N \subset G$ with $N = S \rtimes C$,
where $S$ is a Sylow $2$-subgroup of $G$ and $C \simeq \mathbb{Z}_{q-1}$.
Assume $M \subset N$ and set $S_{0}:=S$ and keep $C$ as above, so
$M \subset S_{0} \rtimes C \subset G$ with $S_{0},C$ concrete in $G$.
Since $H$ and $K$ are $2$-groups of order $q$, we have $H,K \subset S_{0}$.
In the Sylow \(2\)-subgroup \(S_{0}\) of \(G\) every element of order \(2\) is central.
The subgroup generated by the elements of order \(2\) is an elementary abelian subgroup of order \(q\), which is the center \(Z(S_{0})\).
Hence \(Z(S_{0})\) is characteristic in \(S_{0}\) and \(|Z(S_{0})|=q\).
Any subgroup of order \(q\) in \(S_{0}\) is generated by elements of order \(2\), and it equals \(Z(S_{0})\).
Since \(|H|=|K|=q\) and \(H,K \subset S_{0}\), we obtain \(H=Z(S_{0})=K\), which contradicts \(H \ne K\).

\textbf{Case (3-iii):} Suppose $M\subset D_{2(q-1)}$. Then $H\subset D_{2(q-1)}$, which requires $q \mid 2(q-1)$. Since $q-1$ is odd, this means $q \mid 2$. But $q=2^{2n+1} \geq 8$, which is a contradiction.

\textbf{Case (3-iv):} Suppose $M\subset \mathbb{Z}_{m_{\pm}}\rtimes \mathbb{Z}_{4}$. Then $H\subset \mathbb{Z}_{m_{\pm}}\rtimes \mathbb{Z}_{4}$, which requires $q \mid 4m_{\pm}$. Since $m_{\pm}$ is odd, this means $q \mid 4$. But $q \geq 8$, which is a contradiction.

Therefore, $M$ cannot be contained in any proper maximal subgroup of $G$, which contradicts our assumption that $M \subsetneq G$. Hence $M = G$. Since $G$ itself is trivially a pronormal subgroup, we conclude that when $2n+1$ is prime, $\PrN(G)$ is always closed under join.
\end{proof}

\subsection{An alternative meet for pronormal subgroups}

\begin{remark}\label{rem:束まとめ}
For the values of \(q\) considered here, the description of the family \(\PrN(G)\) is common to \(\PSL(2,q)\), \(J_{1}\), and \(\Sz(q)\).
By Propositions \ref{prop:PSLの準正規はmeetで閉じない}, \ref{prop:J1の準正規はmeetで閉じない}, \ref{prop:Szの準正規はmeetで閉じない}, the intersection of two pronormal subgroups need not be pronormal.
By Propositions \ref{prop:PSLの準正規はjoinで閉じる}, \ref{prop:J1の準正規はjoinで閉じる}, \ref{prop:Szの準正規はjoinで閉じている}, the join of two pronormal subgroups is always pronormal.

Define a canonical meet on \(\PrN(G)\) as follows.
For \(H,K\in\PrN(G)\) let \(H\wedge_{\PrN}K\) be the unique largest pronormal subgroup contained in \(H\cap K\).
Since \(G\) is finite, the set \(\{L\in\PrN(G)\mid L\subset H\cap K\}\) has maximal elements.
If two distinct maximal elements \(A\) and \(B\) existed, then \(A\vee B\) would be pronormal and would still lie in \(H\cap K\), which contradicts maximality.
Hence \(H\wedge_{\PrN}K\) is well defined and gives the greatest lower bound of \(H\) and \(K\) inside \(\PrN(G)\). Together with the subgroup join, this operation turns \(\PrN(G)\) into a lattice.
\end{remark}

\section*{Acknowledgements}
We are deeply indebted to Professor Koichi Betsumiya for his constant guidance and supervision throughout this research.

We also thank Wei-Liang Sun for valuable advice, especially for pointing us to relevant prior work on pronormal subgroups.

Their support and encouragement were indispensable to the completion of this work.

\end{document}